\documentclass[12pt]{article}

\usepackage{amssymb}
\usepackage{amsfonts}
\usepackage{amsmath,amsthm}

\date{\today}

\newtheorem{theorem}{Theorem}[section]

\newtheorem{proposition}[theorem]{Proposition}
\newtheorem{corollary}[theorem]{Corollary}
\newtheorem{lemma}[theorem]{Lemma}
\theoremstyle{definition}
\newtheorem{definition}[theorem]{Definition}
\newtheorem{remark}[theorem]{Remark}
\newtheorem{example}[theorem]{Example}

\newcommand{\R}{\mathbb R}
\newcommand{\RR}{\overline{\mathbb R}}
\newcommand{\E}{\mathbb E}
\newcommand{\ld}[3]{#2^{(#1)}_{-} #3}
\newcommand{\lsubd}[2]{\mathrm{\partial}^{(#1)}_{-} #2}
\newcommand{\subd}[2]{{\rm \partial} #1 #2}
\newcommand{\sld}[2]{#1^{\pr\pr}_{-} #2}

\newcommand{\gin}[3]{#2^{[#1]}_{-} #3}
\newcommand{\dini}[1]{#1^{\ell}}
\newcommand{\pr}{\prime}
\newcommand{\eps}{\varepsilon}
\newcommand{\f}{f : \E\to\R\cup\{+\infty\}}

\newcommand{\norm}[1]{\Vert#1\Vert}
\newcommand{\dom}[1]{{\rm dom}\,#1} 

\newcommand{\ph}{\varphi}

\title{Potentialities of Nonsmooth Optimization}

\author{\footnotesize V. I. IVANOV\thanks{Department of Mathematics, Technical University of Varna, Bulgaria. Email: vsevolod.ivanov@tu-varna.bg}}

\date{}

\begin{document}
\maketitle

%
\begin{abstract}
In this paper, we show that higher-order optimality conditions can be obtain for arbitrary nonsmooth function.
We introduce a new higher-order directional derivative and higher-order subdifferential of Hadamard type of a given proper extended real function. This derivative is consistent with the classical higher-order Fr\'echet directional derivative in the sense that both derivatives of the same order coincide if the last one exists.  We obtain necessary and sufficient conditions of order $n$ ($n$ is a positive integer) for a local minimum and isolated local minimum of order $n$ in terms of these derivatives and subdifferentials. We do not require any restrictions on the function in our results. A special class $\mathcal F_n$ of functions is defined and optimality conditions for isolated local minimum of order $n$ for a function $f\in\mathcal F_n$ are derived. The derivative of order $n$ does not appear in these characterizations.  We prove necessary and sufficient criteria such that every stationary point of order $n$ is a global minimizer. We compare our results with some previous ones. 

 \end{abstract}



\section{Introduction}
\setcounter{equation}{0}
\label{s1}

One of the main tasks of nondifferentiable optimization is to extend some optimality conditions to more general classes of nondifferentiable functions.  There are necessary and sufficient conditions in unconstrained optimization  in terms of various generalized derivatives. 
The most of them are of first- and second-order. The  higher-order conditions are rather limited.
Such results were obtained in \cite{aub90,gin02,ban01,hof78,jim08,lin82,pal91,stu86,stu99,war94}. 
Even the conditions of first- and second-order  are satisfied for restricted classes of functions when we apply the known directional derivatives: locally Lipschitz, continuously differentiable, lower semicontinuous, the class C$^{1,1}$, and so on functions. Consider, for example, the lower  Dini directional derivative, which is one of the most simple and popular ones. The higher-order necessary conditions hold for an arbitrary nondifferentiable function, but the sufficient ones do not. Recently,  Bedna\v r\'ik and Pastor \cite{pas08} introduced a new class of $\ell$--stable functions and generalized some  second-order sufficient criteria applying these objects. The class of $\ell$-stable functions includes the functions with locally Lipschitz gradient. 

These facts motivate us to search for a new directional derivative such that higher-order optimality conditions in terms of it are satisfied for an arbitrary  function. In our opinion, the following question is important for nondifferentiable optimization: Are there any directional derivatives such that both the necessary conditions for a local minimum and the sufficient ones are satisfied for any function, not necessarily differentiable. In particular, these derivatives should extend the classical Fr\'echet derivatives.

In this paper, we introduce a new generalized directional derivative of order $n$ ($n$ is a positive integer) such that the necessary conditions for optimality and the sufficient ones hold for arbitrary not necessarily differentiable function. We obtain necessary conditions for a local minimum, sufficient ones for a strict local minimum, and complete characterizations of isolated local minimizers of order $n$ ($n$ is a positive integer) in terms of this derivative.  We derive our criteria for arbitrary proper extended real functions. The convergence in the definition of the derivatives  is of Hadamard type.
We introduce a subdifferential of order $n$ and apply it in the optimality criteria. We additionally prove conditions of order $n$, which are both necessary and sufficient for a given point to be a global minimizer. They concern a new class of invex functions of order $n$.
We introduce another new class of functions that we denote by $\mathcal F_n$, where $n$ is a positive integer such that $n\ge 2$. We obtain necessary and sufficient conditions for a point $\bar x$ to be an isolated minimizer of order $n$ of a function $f\in\mathcal F_n$, which are quite different from the case of an arbitrary function. The derivative and the subdifferential of order $n$ of the given function do not appear in our criteria for an isolated local minimizer of order $n$.  At last, we compare our optimality conditions with some known results. We show that some of the theorems by Huang and Ng \cite{hua94} and Chaney \cite{cha87}, which concern the second-order derivative of Chaney follow from our optimality conditions. The main result of Ben-Tal and Zowe \cite{bt85} also is a consequence of our results. Higher-order necessary and higher-order sufficient conditions for an isolated minimum were obtained by Studniarski \cite{stu86}. In his paper, Studniarski applied some derivatives of Hadamard type, which were introduced by him. Other higher-order derivatives of Hadamard type are studied in \cite{aub90,gin02,ban01}. In contrast of our derivatives, the derivatives in \cite{aub90,gin02,stu86} are not consistent with the classical Fr\'echet derivatives. Some of their conditions for optimality do not concern arbitrary functions like our ones.

\section{Higher-order directional derivatives and subdifferentials of Ha\-da\-mard type}
\label{s2}
\setcounter{equation}{0}
In this paper, we suppose that $\E$ is a real finite-dimensional Euclidean space. Denote by $\R$ the set of reals and $\RR=\R\cup\{-\infty\}\cup\{+\infty\}$.
Let $X$ and $Y$ be two linear spaces and $L(X,Y)$ be the space of all continuous linear operators from $X$ to $Y$. Then denote by $L^1(\E)$ the space $L(\E,\R)$, by $L^2(\E)$ the space $L(\E,L^1(\E))$ and so on. If $n$ is an arbitrary positive integer such that $n>1$, let $L^n(\E)$ be the linear space $L(\E,L^{n-1}(\E))$.  Consider a proper extended real function $\f$, that is a function, which never takes the value $-\infty$. The domain of a proper extended real function is the set:
\[
{\rm dom}\; f:=\{x\in\E\mid f(x)<+\infty\}.
\]

\begin{definition}\label{def1}
The lower Hadamard directional derivative of a function $\f$
at a point $x\in\dom f$ in direction $u\in\E$ is defined as follows:
\[
\ld 1 f(x;u)=\liminf_{t\downarrow 0,u^\pr\to u}\,
t^{-1}[f(x+t u^\pr)-f(x)].
\]
Here $t$ tends to 0 with positive values, and $u^\pr\to u$ implies that the norm $\norm{u^\pr-u}$ approaches $0$.
\end{definition}


\begin{definition}
Recall that the lower Hadamard subdifferential of a function $\f$ at some point
$x\in\dom f$ is defined  by the following relation:
\[
\lsubd 1 f(x)=\{ x^*\in L^1(\E)\mid x^*(u)\le\ld 1 f (x;u)\quad\textrm{for all directions}\quad u\in\E\}.
\]
\end{definition}

We introduce the following definitions:

\begin{definition}
Let $\f$ be an arbitrary proper extended real function.
Suppose that $x^*_1$ is a fixed element from the lower Hadamard subdifferential
$\lsubd 1 f (x)$ at the point $x\in\dom f$. Then the lower
second-order derivative of Hadamard type of $f$ at $x\in\dom f$ in direction $u\in\E$ is
defined as follows:
\[
\ld 2 f(x;x^*_1;u)=\liminf_{t\downarrow 0,u^\pr\to u}\,
2t^{-2}[f(x+t u^\pr)-f(x)-tx^*_1(u^\pr)].
\]
\end{definition}

\begin{definition}
Let $\f$ be an arbitrary proper extended real function. Suppose that $x\in\dom f$, $x^*_1\in\lsubd 1 f(x)$.
The lower second-order Hadamard subdifferential of the function $\f$ at the point
$x\in\dom f$ is defined  by the following relation:
\[
\lsubd 2 f(x;x^*_1)=\{ x^*\in L^2(\E)\mid x^*(u)(u)\le\ld 2 f (x;x^*_1;u)\quad\textrm{for all directions}\quad u\in\E\}.
\]
\end{definition}

\begin{definition}\label{def3}
Let $\f$ be an arbitrary proper extended real function, and $n$ be any positive integer. Suppose that the lower Hadamard subdifferential 
\[
\lsubd i f (x;x_1^*,x_2^*,\dots,x_{i-1}^*),\quad i=1,2,\dots, n-1
\]
of order $i$ at the point $x\in\dom f$ is nonempty and $x^*_i$ is a fixed point from it. 
Then the lower derivative of Hadamard type  of order $n$ of $f$ at $x\in\dom f$ in direction $u\in\E$ is defined as follows:
\[
\ld n f(x;x^*_1,x^*_2,\dots,x^*_{n-1};u)=\liminf_{t\downarrow 0,u^\pr\to u}\,\Delta_n,
\]
where
\[
\Delta_n=
n!\, t^{-n}\, [f(x+t u^\pr)-f(x)-\sum_{i=1}^{n-1}\frac{t^i}{i!}\, x^*_i\underbrace{(u^\pr)(u^\pr)\dots (u^\pr)}_{i-\text{times}}].
\]
This derivative is well defined as element of $\bar\R$, because only the term $f(x+t u^\pr)$ can be infinite in the expression for $\Delta_n$.
\end{definition}

\begin{definition}\label{def4}
Suppose that $\f$ is an arbitrary proper extended function, and $n$ is any positive integer.
Let $x^*_i$ be a fixed point from the lower Hadamard subdifferential
$\lsubd i f (x;x_1^*,x_2^*,\dots,x_{i-1}^*)$, $i=1,2,\dots, n-1$ of order $i$ at the point $x\in\dom f$. Then the lower
subdifferential of Hadamard type  of order $n$ of $f$ at $x\in\dom f$  is defined as follows:
\[
\lsubd n f(x;x^*_1,x^*_2,\dots,x^*_{n-1})=\{ x^*\in L^n(\E)\mid x^*\underbrace{(u)(u)\dots (u)}_{n-\text{times}}
\]
\[
\le\ld n f (x;x^*_1,x^*_2,\dots,x^*_{n-1};u),\;\forall u\in\E\}.
\]
\end{definition}

The essence of the next result  is that the  derivatives, defined in Definition \ref{def3}
generalize the usual classical ones in contrast of the derivative in \cite{stu86} and a lot of other derivatives.

\begin{proposition}\label{pr3}
Let the function $\f$ have Fr\'echet derivatives
\[
\nabla f(y),\nabla^2 f(y),\dots,\nabla^{n-1} f(y)
\]
at each point $y\in\E$ from some neighborhood of the point $x\in\E$, and let
there exists the n-th order Fr\'echet derivative $\nabla^n f(x).$
Then the lower derivatives of order $m$ exist for every integer $m$ such that $1\le m\le n$ and we have the following relations:
\[
\ld 1 f(x;u)=\nabla f(x)(u);\quad
\lsubd 1 f(x)=\{\nabla f(x)\};
\]
\[
\ld m f(x;\nabla f(x),\nabla^2 f(x),\dots,\nabla^{m-1} f(x);u)=\nabla^m f(x)\underbrace{(u)(u)\dots (u)}_{m-\text{times}},\quad m=2,3,\dots, n;
\]
\begin{equation}\label{28}
\nabla^m f(x)\in\lsubd m f(x;\nabla f(x),\nabla^2 f(x),\dots,\nabla^{m-1} f(x)) ,\quad m=2,3,\dots, n.
\end{equation}
\end{proposition}
\begin{proof}
The first-order relations are well known, because they concern the Hadamard directional derivative.

We prove by induction the relations of order $m>1$. Suppose that they are satisfied for every positive integer $k<m$. It follows from here that 
\[
\ld m f(x;\nabla f(x),\nabla^2 f(x),\dots,\nabla^{m-1} f(x);u)
\]
is well defined. By Taylor's expansion formula with a reminder in the form of Peano \cite{il82} we have
\[
f(x+t u^\pr)=f(x)+\sum_{i=1}^m\,\frac{1}{i!}\, \nabla^i f(x)\underbrace{(tu^\pr)(tu^\pr)\dots (tu^\pr)}_{i-\text{times}}]+o(t^m),
\]
where $o(h)$ is a function such that $\lim_{h\to 0}\, o(h)/h=0$.
Then we conclude from Definition \ref{def3} that
\[
\ld m f(x;\nabla f(x),\nabla^2 f(x),\dots,\nabla^{m-1}f(x);u)
\]
\[
=\liminf_{t\downarrow 0,u^\pr\to u}\,[\nabla^m f(x)\underbrace{(u^\pr)(u^\pr)\dots (u^\pr)}_{m-\text{times}}+o(t^m)/t^m]=\nabla^m f(x)\underbrace{(u)(u)\dots (u)}_{m-\text{times}}.
\]
By Definition \ref{def4} we obtain that Inclusions (\ref{28}) are satisfied.
\end{proof}

\section{Conditions for a local minimum}
\label{s3}
\setcounter{equation}{0}
\begin{theorem}\label{th1}
Let $\bar x\in\dom f $ be a local minimizer of the proper extended real function $\f$. Then
\begin{equation}\label{14}
0\in\lsubd 1 f(\bar x),\quad 0\in\lsubd n f(\bar x;\underbrace{0,0,\dots ,0)}_{(n-1)-times}\quad\text{for all}\quad n=2,3,4,\dots
\end{equation}
\end{theorem}
\begin{proof}
Since $\bar x$ is a local minimizer, then there exists a neighbourhood $N\ni\bar x$ with $f(x)\ge f(\bar x)$ for all $x\in N$. Let $u\in\E$ be an arbitrary chosen direction. Then $f(\bar x+t u^\pr)\ge f(\bar x)$ for all sufficiently small positive numbers $t$ and for all directions $u^\pr$, which are  sufficiently close to $u$. It follows from Definition \ref{def1} that $\ld 1 f (\bar x;u)\ge 0$. Therefore $0\in\lsubd 1 f(\bar x)$, because $u\in\E$ is an arbitrary direction.

By the definition of the second-order lower derivative,
using that $0\in\lsubd 1 f(\bar x)$ we obtain that $\ld 2 f(\bar x;u;0)$ is well defined and
\[
\ld 2 f (\bar x;0;u)=\liminf_{t\downarrow 0,u^\pr\to u,}\, 2\, t^{-2}
[f(\bar x+t u^\pr)-f(\bar x)]\ge 0
\]
for all directions $u\in\E$. Therefore $0\in\lsubd 2 f(\bar x;0)$.

Let $n$ be an arbitrary positive integer and 
\[
0\in\lsubd i f(\bar x;\underbrace{0,0,\dots ,0)}_{(i-2)-times},\quad i=1,2,\dots,n-1
\]
Hence $\ld n f(\bar x;\underbrace{0,0,\dots ,0}_{(n-1)-times};u)$ has sense and
\[
\ld n f(\bar x;\underbrace{0,0,\dots ,0}_{(n-1)-times};u)=\liminf_{t\downarrow 0,u^\pr\to u,}\, n!\, t^{-n}
[f(\bar x+t u^\pr)-f(\bar x)]\ge 0,\quad\forall u\in\E.
\]
It follows from the definition of the lower subdifferential of order $n$ that
$0\in\lsubd n f(\bar x;\underbrace{0,0,\dots ,0)}_{(n-1)-times}$. 
\end{proof}

\begin{remark}
Condition (\ref{14}) is equivalent to the following one:
\begin{equation}\label{15}
\ld n f(\bar x;\underbrace{0,0,\dots ,0}_{(n-1)-times};u)\ge 0,\quad\textrm{ for all }u\in\E,\;\textrm{ for all positive integers } n.
\end{equation}
\end{remark}

\begin{corollary}
If $\bar x$ is a local minimizer, then for  every positive integer $n$ there exist $x_1^*$, $x_2^*$,\dots,$x_{n-1}^*$, which do not depend on $u$ such that
\begin{equation}\label{30}
x_1^*\in\lsubd 1 f(\bar x),\; x_i^*\in\lsubd i f(\bar x;x_1^*,x_2^*,\dots, x_{i-1}^*),\; i=2,3,\dots, n
\end{equation}
and
\begin{equation}\label{26}
\ld 1 f(\bar x;u)\ge 0,\quad \ld i f(\bar x;x^*_1,x^*_2,\dots,x^*_{i-1};u)\ge 0,\; i=2,3,\dots,n.
\end{equation}
\end{corollary}
\begin{proof}
We choose $x_1^*=0$, $x_2^*=0$,\dots,$x_{n-1}^*=0$.
\end{proof}

The following example shows that Condition (\ref{15}) is not sufficient for $\bar x$ to be a local minimizer:
\begin{example}\label{ex2}
Consider the function of one variable $f:\R\to\R$ defined by:
\[
f(x)=\left\{
\begin{array}{ll}
-\exp\,(-1/x^2), & \textrm{if}\quad x\ne 0, \\
0, & \textrm{if}\quad x=0.
\end{array}\right.
\]
Let us take $\bar x=0$. Then we have
\[
\ld n f(\bar x;\underbrace{0,0,\dots ,0}_{(n-1)-times};u)=0\quad\textrm{ for all }u\in\R,\;\textrm{ for all positive integers } n,
\]
\[
\lsubd n f(\bar x;\underbrace{0,0,\dots ,0)}_{(n-1)-times}=\{0\}\;\textrm{if }n\textrm{ is odd,}\quad
\lsubd n f(\bar x;\underbrace{0,0,\dots ,0)}_{(n-1)-times}=(-\infty,0]\;\textrm{if }n\textrm{ is even.}
\]
Hence Condition {\rm (\ref{15})} is satisfied, but $\bar x$ is not a local minimizer. Really, it is a global maximizer.
\end{example}

On the other hand the following sufficient conditions hold:
\begin{theorem}\label{th4}
Let be given a proper extended real function $\f$ and a point $\bar x\in\dom f$. Suppose that for every direction $u\in\E$, $u\ne 0$ we have $\ld 1 f(\bar x;u)>0$,
or there exists a positive integer $n=n(u)$, $n\ge 2$, which depend on $u$, and such that the following conditions hold:
\begin{equation}\label{4}
0\in\lsubd 1 f(\bar x),\quad 0\in\lsubd i f(\bar x;\underbrace{0,0,\dots ,0)}_{(i-1)-times}\quad\text{for all}\quad i=1,2,\dots, n-1
\end{equation}
and
\begin{equation}\label{5}
\ld {n(u)} f(\bar x;\underbrace{0,0,\dots ,0}_{(n-1)-times};u)>0.
\end{equation}
Then $\bar x$ is a strict local minimizer.
\end{theorem}
\begin{proof}
Let $u\ne 0$ be an arbitrary direction. It follows from (\ref{4}) and (\ref{5}) that there exists $\alpha>0$ with
\[
\liminf_{t\downarrow 0,u^\pr\to u}\,n!\,t^{-n}[f(\bar x+t u^\pr)-f(\bar x)]>2\alpha>0.
\]
Therefore, there exist $\delta>0$ and $\eps>0$ such that
\begin{equation}\label{29}
f(\bar x+t u^\pr)\ge f(\bar x)+\alpha\, t^n/n!>f(\bar x)
\end{equation}
for every $t\in (0,\delta)$ and arbitrary $u^\pr$ with $\norm{u^\pr-u}<\eps$.

Without loss of generality we may suppose that $u$ belongs to the unit sphere $S:=\{u\in\E\mid\norm{u}=1\}$. Since $u$ is arbitrary chosen,  then we can cover $S$ by  neighbourhoods $N(u;\eps):=\{u^\pr\in S\mid\norm{u^\pr-u}<\eps\}$ such that (\ref{29}) is satisfied. Taking into account that the unit sphere is compact, then we can choose a finite number of neighbourhoods $N(u_1;\eps_1)$, $N(u_2,\eps_2)$,\dots $N(u_s;\eps_s)$ that cover $S$. Let the respective values of $\delta$ are $\delta_1$, $\delta_2$,\dots $\delta_s$ and $\bar \delta=\min\{\delta_i\mid 1\le i\le s\}$. Then we have
\[
f(\bar x+t u^\pr)>f(\bar x),\quad\forall u^\pr\in S,\;\forall t\in(0,\bar\delta). 
\]
Hence, $f(x)>f(\bar x)$ for all $x\in\E$ such that $\norm{x-\bar x}<\bar\delta$, which implies that $\bar x$ is a strict local minimizer.
\end{proof}

\section{Isolated minimizers and optimality conditions}
\label{s4}
\setcounter{equation}{0}
The following definition was introduced by Studniarski \cite{stu86} as a generalization of the respective notion of order $1$ and $2$ in \cite{aus84}.

\begin{definition}
Let $n$ be a positive integer.
A point $\bar x\in\dom f$ is called an isolated local minimizer of
order $n$ for the function $\f$ iff there exist a neighbourhood $N$
of $\bar x$ and a constant $C>0$ with
\begin{equation}\label{1}
f(x)\ge f(\bar x)+C\norm{x-\bar x}^n,\quad\forall  x\in N.
\end{equation}
\end{definition}

\begin{theorem}\label{th2}
Let be given a proper extended real function $\f$ and $\bar x\in\dom f$. Then $\bar x$ is an isolated local minimizer of order $n$, where $n$ is a positive integer such that $n\ge 2$, if and only if (\ref{4}) is satisfied
and
\begin{equation}\label{35}
\ld {n} f(\bar x;\underbrace{0,0,\dots ,0}_{(n-1)-times};u)>0,\quad\forall\; u\in\E\setminus\{0\}.
\end{equation}
\end{theorem}
\begin{proof}
Let $\bar x$ be an isolated local minimizer of order $n$. We prove that Conditions (\ref{4}) and (\ref{35}) hold.
Suppose that $u\in\E$ is arbitrary chosen. It follows from Inequality (\ref{1}) that there exist numbers $\delta>0$, $\varepsilon>0$ and $C>0$ with
\begin{equation}\label{3}
f(\bar x+tu^\pr)\ge f(\bar x)+Ct^n\norm{u^\pr}^n
\end{equation}
for all  $t\in (0,\delta)$ and every $u^\pr$ such that $\norm{u^\pr-u}<\varepsilon$. Therefore 
\begin{equation}\label{2}
\ld 1 f (\bar x;u)=\liminf_{t\downarrow 0,u^\pr\to u,}\,t^{-1}[f(\bar x+t u^\pr)-f(\bar x)]\ge\liminf_{t\downarrow 0,u^\pr\to u,}\, Ct^{n-1}\norm{u^\pr}^{n}=0
\end{equation}
if $m>1$, and $0\in\ld 1 f(\bar x)$ if $m=1$.
Therefore $0\in\lsubd 1 f(\bar x)$. 

Suppose that $m$ is an arbitrary positive integer such that $1\le m<n$ and we have
\[
0\in\lsubd i f(\bar x;\underbrace{0,0,\dots ,0)}_{(i-1)-times}\quad\text{for all}\quad i<m.
\]
It follows from (\ref{3}) that
\[
\ld m f(\bar x;\underbrace{0,0,\dots ,0}_{(m-1)-times};u)=\liminf_{t\downarrow 0,u^\pr\to u,}\,m!\, t^{-m}[f(\bar x+t u^\pr)-f(\bar x)]
\ge\liminf_{t\downarrow 0,u^\pr\to u,}\, Ct^{n-m}\norm{u^\pr}^{n}\ge 0,\;\forall u\in\E.
\]
and $0\in\lsubd m f(\bar x;\underbrace{0,0,\dots ,0)}_{(m-1)-times}$. Then it follows from (\ref{3}) that
\[
\ld n f(\bar x;\underbrace{0,0,\dots ,0}_{(n-1)-times};u)\ge\liminf_{t\downarrow 0,u^\pr\to u,}\, C\norm{u^\pr}^{n} >0,\quad\forall u\in\E\setminus\{0\}.
\] 

Conversely, suppose that Conditions {\rm (\ref{4})} and {\rm (\ref{35})}  hold. We prove that $\bar x$ is an isolated local minimizer of order $n$.
Assume the contrary that $\bar x$ is not an isolated minimizer of order $n$.
Therefore, for every sequence $\{\varepsilon_k\}_{k=1}^\infty$ of positive numbers
converging to zero, there exists a sequence $\{x_k\}$ with $x_k\in\dom f$ such that
\begin{equation}\label{6}
\norm{x_k-\bar x}\le \varepsilon_k,\quad
f(x_k)< f(\bar x)+\varepsilon_k\norm{x_k-\bar x}^n,
\end{equation}

It follows from (\ref{6}) that $x_k\to\bar x$. Denote $t_k=\norm{x_k-\bar x}$,
$d_k=(x_k-\bar x)/t_k$. Passing to a subsequence, we may suppose
that $d_k\to  d$ where $\norm{d}=1$. It follows from here that
\[
\ld 1 f (\bar x;d)\le\liminf_{k\to\infty}\, t^{-1}_k[f(\bar x+t_k d_k)-f(\bar x)]
=\liminf_{k\to\infty}\, t^{-1}_k[f(x_k)-f(\bar x)]
\le\liminf_{k\to\infty}\,\varepsilon_k t_k^{n-1}=0.
\]
It follows from $0\in\subd f(\bar x)$ that $\ld 1 f (\bar x;d)=0$. 

Let $m$ be any integer with $1<m\le n$ such that $\ld i f(\bar x;\underbrace{0,0,\dots ,0}_{(i-1)-times};d)=0$ for $i<m$. Therefore
\begin{equation}\label{16}
\ld {m} f(\bar x;\underbrace{0,0,\dots ,0}_{(m-1)-times};d)=\liminf_{t\downarrow 0,d^\pr\to d,}\,m!\,t^{-m}[f(\bar x+t d^\pr)-f(\bar x)] 
\le\liminf_{k\to\infty}\, m!\,\varepsilon_k\, t_k^{n-m}=0,
\end{equation}
because $n-m\ge 0$ and $\varepsilon_k\to 0$. Then it follows from (\ref{4}) that
\[
\ld m f(\bar x;\underbrace{0,0,\dots ,0}_{(m-1)-times};d)=0\quad {\rm if}\quad m<n.
\]
 We conclude from the case $m=n$ that Inequality (\ref{16}) contradicts Condition (\ref{35}).
\end{proof}

\begin{example}
Consider the function $f:\R^2\to\R$ defined by
\[
f(x_1,x_2)=\left\{
\begin{array}{ll}
\exp\,(-1/(x_1^2+x_2^2)), & \textrm{if}\quad (x_1,x_2)\ne (0,0), \\
0, & \textrm{if}\quad (x_1,x_2)=(0,0).
\end{array}\right.
\]
The point $\bar x=(0,0)$ is a strict global minimizer, but there is no a positive integer $n$ such that $\bar x$ is an isolated minimizer of order $n$. 
\end{example}

\section{Global optimality conditions with a higher-order invex function}
\label{s5}
\setcounter{equation}{0}
Example \ref{ex2} shows that the necessary conditions for a local minimum are not sufficient for a global one.
Then the following question arises: Which is the largest class of functions such that the necessary optimality conditions  from Theorem \ref{th1} become sufficient for a global minimum. 
Recently, Ivanov \cite{optimization} introduced a new class of Fr\'echet differentiable functions called second-order invex ones in terms of the usual second-order directional derivative. They extend the so called invex ones and obey the following property: A Fr\'echet differentiable function is second-order invex if and only if each second-order stationary point is a global minimizer. We extend the notions invexity and second-order invexity to nondifferentiable functions in terms of the lower Hadamard directional derivatives of order $n$. Some more developments to inequality constrained problems in terms of the usual second-order directional derivative are recently obtained by Ivanov \cite{jogo2011}.

First, we recall the definition of an invex function \cite{han81} in terms of the lower Hadamard directional derivative.

\begin{definition}
 A proper extended real function $\f$ is called invex in terms of the lower Hadamard directional derivative iff there exists a map
$\eta_1: \E\times \E\to\E$  such that the following inequality holds for all $x\in\E $, $y\in\E$: 
\begin{equation}\label{13}
f(y)-f(x)\ge\ld 1 f(x;\eta_1(x,y)). 
\end{equation}
\end{definition}

We introduce the following two definitions:
\begin{definition}
We call a proper extended function $\f$  invex of order $n$ 
in terms of the lower Hadamard derivatives  iff for every $\bar x\in\dom f$, $x\in\E$ such that there exist at least one $(i-1)$-ple $(x_1^*,x_2^*,\dots,x_{i-1}^*)$ with
\begin{equation}\label{7}
x_1^*\in\lsubd 1 f(\bar x),\; x_i^*\in\lsubd i f(\bar x;x_1^*,x_2^*,\dots, x_{i-1}^*),\; i=2,3,\dots, n
\end{equation}
there are $\eta_1$, $\eta_2,\dots$, $\eta_n$, which depend on $\bar x$ and $x$ such that the  following inequality holds 
\begin{equation}\label{8}
f(x)-f(\bar x)\ge\ld 1 f(\bar x;\eta_1(\bar x,x))+\sum_{i=2}^n\ld i f(\bar x; x^*_1,x_2^*,\dots,x_{i-1}^* ;\eta_i(\bar x,x))
\end{equation}
for all $x_i^*$, $i=1,2,\dots,n$ satisfying Conditions (\ref{7}).

If there exist $\eta_1(\bar x,x)$, $\eta_2(\bar x,x)$, $\eta_3(\bar x,x),\dots$ such that (\ref{7}) and (\ref{8}) are satisfied with $n=+\infty$, then we call $f$ invex in generalized sense (or invex of order $+\infty$).
\end{definition}

\begin{definition}
Let $\f$ be a given proper extended real function and $n$ be a positive integer. We call a stationary point of order $n$ 
every point $\bar x\in\dom f$ which satisfies the necessary optimality conditions (\ref{30}) and (\ref{26}).

The notion of a 1-stationary point coincides with the notion of a stationary point.
If (\ref{30}) and (\ref{26}) are satisfied for every $n=1,2,3,\dots$, then we call $\bar x$ stationary point in generalized sense (or  order $+\infty)$.
\end{definition}

\begin{theorem}\label{npi}
Let $n$ be a positive integer or $+\infty$ and  $\f$ a proper extended function.  Then $f$ is invex of order $n$ if and only if each stationary point $\bar x\in\dom f$ of order $n$ $(n<+\infty$ or $n=+\infty)$ is a global minimizer of $f$.
\end{theorem}
\begin{proof}
We prove the case $n<+\infty$. The other case is similar.
Suppose that $f$ is invex of order $n$. If the function has no stationary points, then obviously every stationary point is a global minimizer. Suppose that the function has at least one stationary point, that is a point satisfying the necessary optimality conditions (\ref{30}) and (\ref{26}).
Suppose that $\bar x\in\dom f $ is a given stationary point of order $n$. We prove that it is a global minimizer of $f$.
Suppose that  $x$ is an arbitrary point from $\E$. It follows from invexity of order $n$ that there exist $\eta_i(\bar x,x)$, $i=1,2,\dots, n$ such that 
\begin{equation}\label{9}
f(x)-f(\bar x)\ge\ld 1 f(\bar x;\eta(\bar x,x))+\sum_{i=2}^n\ld i f(\bar x; x^*_1,x_2^*,\dots,x_{i-1}^*;\eta_i(\bar x,x) )
\end{equation}
for all $x_i^*$, $i=1,2,\dots,n-1$
Since $\bar x$ is a stationary point of order $n$, then there exist $x_1^*$, $x_2^*$,\dots, $x_{n-1}^*$ with
\[
\ld 1 f(\bar x;u)\ge 0,\;\ld i f(\bar x; x_1^*,x_2^*,\dots ,x_{i-1}^*;u)\ge 0\quad\text{for all}\quad i=2,3,\dots,n,\;\forall u\in\E.
\]
Hence
\[
\ld 1 f(\bar x;\eta_1(\bar x,x))\ge 0,\;\ld i f(\bar x; x_1^*,x_2^*,\dots ,x_{i-1}^*;\eta_i(\bar x,x) )\ge 0\quad\text{for all}\quad i=2,3,\dots,n.
\]
It follows from (\ref{9}) that $f(x)\ge f(\bar x)$. Therefore $\bar x$ is a global minimizer.

Conversely, suppose that every stationary point of order $n$ is a global minimizer. We prove that $f$ is invex of order $n$. 
Assume the contrary.  Hence, there exists a pair $(\bar x,x)\in\dom f\times \E$ such that for every $\eta_i\in\E$, $i=1,2,\dots,n$ there are $x_1^*$, $x_2^*$ , $\dots$, $x_{n-1}^*$ satisfying Conditions (\ref{7}) and
\begin{equation}\label{10}
f(x)-f(\bar x)<\ld 1 f(\bar x;\eta_1)+\sum_{i=2}^n\ld i f(\bar x; x^*_1,x_2^*,\dots,x_{i-1}^*;\eta_i ).
\end{equation}

First, we prove that $f(x)<f(\bar x)$. Let us choose in (\ref{10}) $\eta_i=0$, $i=1,2,\dots,n$. We have 
\[
\ld 1 f(\bar x;0)\le\liminf_{t\downarrow 0}\,t^{-1}(f(\bar x+t.0)-f(\bar x))=0.
\]
Let $i$ be an arbitrary integer such that $1<i\le n$. Then
\[
\ld i f(\bar x;x^*_1,x^*_2,\dots,x^*_{i-1};0)\le \liminf_{t\downarrow 0,u^\pr\to 0}\,\frac{i!}{t^i}\,[f(\bar x+t.u^\pr)-f(\bar x)-\sum_{k=1}^{i-1}\frac{t^k}{k!}x^*_k\underbrace{(u^\pr)(u^\pr)\dots (u^\pr)}_{k-\text{times}}]
\]
\[
\le\liminf_{t\downarrow 0}\,\frac{i!}{t^i}\,[f(\bar x+t.0)-f(\bar x)- \sum_{k=1}^{i-1}\frac{t^k}{k!}x^*_k\underbrace{(0)(0)\dots (0)}_{k-\text{times}}]=0,\quad i=2,3,\dots, n.
\] 
 It follows from (\ref{10}) that $f(x)<f(\bar x)$.

Second, we prove that  
\begin{equation}\label{11}
\ld 1 f(\bar x;u)\ge 0,\quad\forall u\in\E.
\end{equation} 
Suppose the contrary that there exists at least one point $v\in\E$ with $\ld 1 f(x;v)<0$. The lower Hadamard directional derivative is positively homogeneous with respect to the direction, that is 
\[
\ld 1 f(\bar x;\tau u)=\tau\ld 1 f(\bar x;u),\quad\forall \bar x\in\dom f,\;\forall u\in\E,\;\forall \tau\in(0,+\infty).
\]
Then inequality (\ref{10}) is satisfied when $\eta_1=tv$, $t>0$, $\eta_i=0$, $i\ne 1$, that is 
\[
f(x)-f(\bar x)<t\ld 1 f(\bar x;v),\quad\forall t>0,
\]
which is impossible, because $f(x)-f(\bar x)$ is finite and $\ld 1 f(x;v)<0$.  
Therefore, $\ld 1 f(\bar x;u)\ge 0,\quad\forall u\in\E$.

Third, we prove that for all $u\in\E$ there are $x_1^*$, $x_2^*$,\dots,$x_{n-1}^*$ with
\begin{equation}\label{12}
\ld i f(\bar x;x^*_1,x^*_2,\dots,x^*_{i-1};u)\ge 0\quad\textrm{ for } i=2,3,\dots, n.
\end{equation}
Suppose the contrary that there exists $v\in\E$ with $\ld i f(\bar x;x^*_1,x^*_2,\dots,x^*_{i-1};v)<0$ for all  $x^*_1$, $x^*_2$,...,$x^*_{i-1}$ satisfying Conditions (\ref{7}). The lower Hadamard directional derivative of order $i$ is positively homogeneous of degree $i$ with respect to the direction, that is
\[
\ld i f(\bar x;x_1^*,x_2^*,\dots,x_{i-1}^*;tv)=t^i\ld i f(\bar x;x_1^*,x_2^*,\dots,x_{i-1}^*;v),\quad\forall t>0.
\]
Then it follows from (\ref{10}) with $\eta_i=tv$, $t>0$, $\eta_k=0$ when $k\ne i$ that
\[
f(x)-f(\bar x)<t^i\ld i f(\bar x;x_1^*,x_2^*,\dots,x_{i-1}^*;v ),\quad\forall t>0,
\]
which is impossible when $t$ is sufficiently large positive number.

The following is the last part of the proof.   It follows from (\ref{11}) and (\ref{12}) that $\bar x$ is a stationary point of order $n$. According to the hypothesis $\bar x$ is a global minimizer, which contradicts the inequality $f(x)<f(\bar x)$. 
\end{proof}

In the next claim we show that the class of invex functions of order $(n+1)$ contains all invex functions of order $n$ in terms of the lower Hadamard directional derivative.

\begin{proposition}
Let $\f$ be an invex function of order $n$. Then $f$ is invex of order $(n+1)$. Every invex function of order $n$ is invex of order $+\infty$.
\end{proposition}
\begin{proof}
It follows from Equation (\ref{8})  that $f$ is invex of order $(n+1)$ keeping the same maps $\eta_1$, $\eta_2$,..., $\eta_n$ and taking $\eta_{n+1}=0$, because
$\ld {n+1} f(\bar x;x^*_1,x^*_2,\dots,x^*_n;0)\le 0$ for all elements of the respective subdifferentials $x^*_1$, $x^*_2$,...,$x^*_n$.
\end{proof}

The converse claim is not satisfied. There are a lot of second-order invex functions, which are not invex. The following example is extremely simple.
\begin{example}
Consider the function $f:\R^2\to\R$ defined by 
\[
f(x_1,x_2)=-x_1^2-x_2^2.
\]
We have $\ld 1 f(x;u)=-2x_1 u_1-2x_2 u_2$ where $u=(u_1,u_2)$ is a direction. Its only stationary point is $\bar x=(0,0)$. This point is not a global minimizer. Therefore, the function is not invex. We have $\lsubd 1 f(\bar x)\equiv \{(0,0)\}$ and $\ld 2 f(\bar x;0;u)=-2u_1^2-2u_2^2$. It follows from here that $f$ has no second-order stationary points. Hence, every second-order stationary point is a global minimizer, and the function is second-order invex.
\end{example}

\begin{example}\label{npc}
Consider the function $f_n:\R\to\R$, where $n\ge 2$ is a positive integer:
\[
f_n(x)=\left\{
\begin{array}{rr}
x^n\,, & x\ge 0\,, \\
(-1)^{n-1}x^n\,, & x<0\,.
\end{array}\right.
\]
If $n$ is an odd number, then $f_n=x^n$. For $n$ even $f_n$ is a function from the class C$^{n-1}$ but not from the class C$^n$.
It has no stationary points of order $n$. Therefore, every stationary point of order $n$ is a global minimizer. According to Theorem \ref{npi}, it is invex of order $n$. On the other hand, the point $x=0$ is stationary of order $(n-1)$. Taking into account that $x=0$ is not a global minimizer, we conclude from the same theorem that the function is not invex of order $(n-1)$.
 
\end{example}

\section{Second-order conditions for a special class of functions in terms of the lower Dini derivatives}
\label{s6}
\setcounter{equation}{0}
In this section, we derive optimality conditions for an isolated minimum of
order two for a special class of functions.

Strongly pseudoconvex functions were introduced by Diewert, Avriel and Zang \cite{d-a-z81}. Their definition assumes
additionally strict pseudoconvexity. It was proved by Hadjisavvas and Schaible
\cite{had93} that in the differentiable case, strict pseudoconvexity of the function is superfluous; in other words each function, which satisfies the next definition is strictly pseudoconvex.

\begin{definition}[\cite{had93}]\label{def2}
Let $S$ be an open convex subset of $\E$. A Fr\'echet differentiable function
$f:S\to\R$ is said to be strongly pseudoconvex iff, for all $x\in S$,
$u\in\E$ such that $\norm{u}=1$ and $\nabla f(x)(u)=0$, there exist positive
numbers $\delta$ and $\alpha$ with $x+\delta u\in S$ and
\[
f(x+t u)\ge f(x)+\alpha t^2,\quad 0\le t<\delta.
\]
\end{definition}

Recall the definition of a strongly pseudoconvex function in terms of lower Dini directional derivative, which were introduced by Biancki \cite{bia96}.

\begin{definition}
The lower Dini directional derivative of a function $\f$ at the point $x\in\dom f$ in
direction $u\in\E$ is defined as follows:
\[
\dini f(x;u)=\liminf_{t\downarrow 0}\, t^{-1}[f(x+tu)-f(x)].
\]
\end{definition}

We adopt the next two definitions to proper extended real functions:

\begin{definition}[\cite{bia96}]\label{DefBia96}
A proper extended function $\f$ is said to be strongly pseudoconvex iff, 
\begin{enumerate}
\item[{\rm (i)}] $f$ is strictly pseudoconvex, that is for all $x\in\dom f $, $u\in\E$ such that $\norm{u}=1$ the condition 
\[
\dini f(x;u)\ge 0\quad {\rm implies}\quad f(x+tu)>f(x),\;\forall\; t>0;
\]
\item[{\rm (ii)}] if $\dini f(x;u)=0$, then there exist positive
numbers $\delta$ and $\alpha$ with
\[
f(x+t u)\ge f(x)+\alpha t^2,\quad\forall t\in[0,\delta).
\]
\end{enumerate}
\end{definition}

It was shown by example in \cite{bia96} that Condition (i) cannot be omitted in Definition \ref{DefBia96}, that is Condition (ii) does not imply Condition (i).

The following notion extends the Lipschitz continuity of the gradient, which were applied in the sufficient conditions due to Ben-Tal and Zowe \cite{bt85}:

\begin{definition}[\cite{pas08}]
A proper extended function $\f$ is called $\ell$-stable at the point $x\in\dom f$ iff there exist a neigbourhood $U$ of $x$ and a constant $K>0$ such that 
\[
|\dini f(y;u)-\dini f(x;u)|\le K\,\norm{y-x}\,\norm{u},\quad\forall y\in U\cap\dom f,\;\forall u\in\E.
\]
\end{definition}

\begin{proposition}[\cite{pas08}]\label{pr1}
Let the proper extended function $\f$ be continuous near $x\in\dom f$ and $\ell$-stable at $x$. Then $f$ is strictly differentiable at $x$, hence Fr\'echet differentiable at $x$.
\end{proposition}

The following mean-value theorem is due to Diewert \cite{die81}.
\begin{lemma}\label{mv}
Let $\ph:[a,\,b]\to\R\cup\{+\infty\}$ be a lower semicontinuous function of one real variable with $a\in\dom \ph$.
Then there exists $\xi$, $a<\xi\le b$, such that 
\[
\ph(a)-\ph(b)\le \dini \ph(\xi; a-b).
\]
\end{lemma}

In the next theorem, we derive necessary and sufficient conditions for an isolated local minimum of second-order of a function, which satisfies Condition (ii) from Definition \ref{DefBia96} at the same given point $\bar x$:

\begin{theorem}\label{th3}
Let the proper extended function $\f$ be continuous near $\bar x\in\dom f$ and $\ell$-stable at $\bar x$.
 Suppose that  $f$ satisfies Condition {\rm (ii)} from Definition \ref{DefBia96} only at the point $\bar x$. 
Then $\bar x$ is an isolated local minimizer of second-order if and only if $\nabla f(\bar x)=0$.
\end{theorem}
\begin{proof}
Let $\bar x$ be an isolated local minimizer of second-order.
We conclude from Proposition \ref{pr1} that $\nabla f(\bar x)$ exists. Then it is obvious that  $\nabla f(\bar x)=0$.

We prove the converse claim. Suppose that $\nabla f(\bar x)=0$, but
$\bar x$ is not an isolated local minimizer of second-order. Therefore,
for every sequence $\{\varepsilon_k\}_{k=1}^\infty$ of positive numbers
converging to zero, there exists a sequence $\{x_k\}$, $x_k\in\dom f$ such that
\begin{equation}\label{27}
\norm{x_k-\bar x}\le \varepsilon_k,\quad
f(x_k)< f(\bar x)+\varepsilon_k\norm{x_k-\bar x}^2.
\end{equation}
It follows from here that $x_k\to \bar x$. Denote $t_k=\norm{x_k-\bar x}$,
$d_k=(x_k-\bar x)/t_k$.  Passing to a subsequence, we may suppose
that the sequence $\{d_k\}_{k=1}^\infty$ is convergent and $d_k\to  d$,  where $\norm{d}=1$.
Therefore
\begin{equation}\label{25}
\liminf_{k\to\infty}\, t^{-2}_k
[f(\bar x+t_k d_k)-f(\bar x)]=\liminf_{k\to \infty}\, t^{-2}_k[f(x_k)-f(\bar x)]\le\lim_{k\to \infty}\, \varepsilon_k=0.
\end{equation}
We have
\[
f(\bar x+t_k d)-f(\bar x)=[f(\bar x+t_k d)-f(\bar x+t_kd_k)]+[f(\bar x+t_kd_k)-f(\bar x)],
\]
since $f(x_k)$ is finite by (\ref{27}).
It follows from Diewert's mean-value theorem that there exists $\theta_k\in[0,1)$ with
\[
f(\bar x+t_k d)-f(\bar x+t_k d_k)\le D_{-} f(y_k;t_k(d-d_k))=t_k\, \dini f(y_k;d-d_k)
\]
where $y_k=\bar x+t_k d_k+t_k\theta_k(d-d_k)$.
Since $\nabla f(\bar x)=0$ and $f$ is $\ell$-stable at $\bar x$, then there exists $K>0$ such that
\[ 
|\dini f(y_k;d-d_k)|\le K\,\norm{y_k-\bar x}\,\norm{d-d_k}
\]
for all sufficiently large integers $k$. 
Taking into account that $d_k\to d$ when $k\to\infty$, we conclude  that
\[
\liminf_{k\to\infty}\, t^{-2}_k
[f(\bar x+t_k d)-f(\bar x)]\le 0.
\]
On the other hand, according to the hypothesis   $f$ satisfies Condition (ii) from Definition \ref{DefBia96}. Therefore
\[
f(\bar x+t_k d)\ge f(\bar x)+\alpha t_k^2
\]
for all sufficiently large $k$. Hence,
\[
\liminf_{k\to\infty}\, t^{-2}_k
[f(\bar x+t_k d)-f(\bar x)]\ge\liminf_{k\to\infty}\,\alpha=\alpha>0,
\]
which is a contradiction.
\end{proof}

The following example shows that Theorem \ref{th3} is not true for functions, which are not $\ell$-stable:

\begin{example}\label{ex1}
Consider the function 
\[
f=|\,x_2-\sqrt[3]{x_1^4}\,|^{\, 3/2}.
\]
Of course, the point $\bar x=(0,0)$ is a local and global minimizer, but it is not an isolated local minimizer of order two. Even it is not a strict local minimizer, because $f(x)=0$ for all $x=(x_1,x_2)$ over the curve $x_2=x_1^{4/3}$. We have $\nabla f(\bar x)=(0,0)$.
Simple calculations show that this function satisfy Condition {\rm (ii)} from Definition \ref{DefBia96} at $\bar x$.  Let $v=(v_1,v_2)$ be an arbitrary vector, whose norm is $1$. If $v_2>0$ or $v_2<0$, then 
\[
\lim_{t\downarrow 0}\,[f(\bar x+tv)-f(\bar x)]/t^2=\lim_{t\downarrow 0}\, f(tv)/t^2=+\infty.
\]
 If $v_2=0$, then $v_1=\pm 1$ and 
\[
\lim_{t\downarrow 0}\,[f(\bar x+tv)-f(\bar x)]/t^2=\lim_{t\downarrow 0}\, f(tv)/t^2=1.
\]
Therefore, for every $v\in\R^2$ there exists $\delta>0$ and $C>0$ such that 
\[
f(tv)>Ct^2 \quad\textrm{for all}\quad t\in(0,\delta).
\]
The sufficient conditions of Theorem \ref{th3} are not satisfied, because $f$ is not $\ell$-stable at $\bar x$. Indeed,
if we take
$x^\pr_k=(0,k^{-1})$, then $\nabla f(x_k)=(0,3/2k^{-1/2})$ and there do not exist $K>0$ such that 
\[
\norm{\nabla f(x^\pr_k)-\nabla f(\bar x)}\le K\norm{x^{\pr}_k-\bar x}
\]
for all sufficiently large integers $k$. We have
\[
\lim_{k\to+\infty}\norm{\nabla f(x^{\pr}_k)}\, / \,
\norm{x^{\pr}_k}=+\infty.
\]
\end{example}

\section{Higher-order conditions for a special class of functions in terms of the lower Hadamard derivatives}
\label{s7}
\setcounter{equation}{0}
We introduce the following notion:
\begin{definition}
We say that a proper extended real function $\f$ belongs to the class $\mathcal F_2(x)$  iff, $x\in\dom f$ and for every $u\in\E$ such that
$\norm{u}=1$, $\ld 1 f(x;u)=0$,
there exist positive numbers $\eps$, $\delta$, and $\alpha$ which satisfy the inequality
\begin{equation}\label{19}
f(x+tu^\pr)\ge f(x)+\alpha t^2,
\end{equation}
for all $t\in\R$ and $u^\pr\in\E$ with $0\le t<\delta$, $\norm{u^\pr-u}<\varepsilon$. 
\end{definition}
The class containing all functions $f$ such that $f\in\mathcal F_2(x)$ for every $x\in\E$, and which are additionally strictly pseudoconvex coincides with the set of all strongly pseudoconvex functions with respect to the lower Hadamard directional derivative.

\begin{definition}
Let $x\in\E$. For any positive integer $n\ge 3$, we say that a proper extended real function $\f$ belongs to the class $\mathcal F_n(x)$  iff, $x\in\dom f$ and  for every $u\in\E$ such that $\norm{u}=1$, 
\begin{equation}\label{22}
0\in\lsubd 1 f(x),\quad 0\in\lsubd i f (x;x_1^*,x_2^*,\dots,x_{i-1}^*),\; i=2,3,\dots, n-2,
\end{equation}
\begin{equation}\label{23}
\ld 1 f(x;u)=0,\quad \ld i f(x;\underbrace{0,0,\dots ,0}_{(i-1)-times};u)=0,\; i=2,3,\dots,n-1,
\end{equation}
there exist positive numbers $\eps$, $\delta$, and $\alpha$ which satisfy the inequality
\begin{equation}\label{24}
f(x+tu^\pr)\ge f(x)+\alpha t^n,
\end{equation}
for all $t\in\R$ and $u^\pr\in\E$ with $0\le t<\delta$, $\norm{u^\pr-u}<\varepsilon$. 
\end{definition}

\begin{theorem}\label{th5}
Let $n\ge 2$ be a positive integer, $\bar x\in\E$, and $f\in\mathcal F_n(\bar x)$.
Then $\bar x$ is an isolated local minimizer of order $n$ if and only if 
\begin{equation}\label{20}
0\in\lsubd 1 f(\bar x),\quad 0\in\lsubd i f (\bar x;\underbrace{0,0,\dots ,0}_{(i-1)-times}),\; i=2,3,\dots, n-1.
\end{equation}
\end{theorem}
\begin{proof}
It follows from Theorem \ref{th1} that the condition $\bar x$  is an isolated local minimizer
of order $n$ implies that Inclusions (\ref{20}) are satisfied.

Conversely, suppose that Inclusions (\ref{20}) are satisfied. We prove that $\bar x$ is an isolated local  minimizer of order $n$.
Assume the contrary.  Therefore, for every sequence $\{\varepsilon_k\}_{k=1}^\infty$ of positive numbers
converging to zero, there exists a sequence $\{x_k\}$, $x_k\in\dom f$ such that
\begin{equation}\label{21}
\norm{x_k-\bar x}\le \varepsilon_k,\quad
f(x_k)< f(\bar x)+\varepsilon_k\norm{x_k-\bar x}^n,
\end{equation}

It follows from (\ref{21}) that $x_k\to\bar x$. Denote $t_k=\norm{x_k-\bar x}$,
$d_k=(x_k-\bar x)/t_k$. Passing to a subsequence, we may suppose without loss of generality
that $d_k\to  d$,  where $\norm{d}=1$. 

Let $m$ be an arbitrary positive integer with $1\le m\le n$.
It follows from (\ref{20}) that
\[
\ld m f (\bar x;\underbrace{0,0,\dots ,0}_{(m-1)-times};d)=\liminf_{k\to\infty}\,m!\, t^{-m}_k[f(\bar x+t_k d_k)-f(\bar x)]\le\liminf_{k\to 0}\, m!\, t^{n-m}\varepsilon_k=0
\]
if $m>1$ or $\ld 1 f(\bar x;d)\le 0$ if $m=1$.
If $1<m<n$, then we conclude from 
\[
0\in\lsubd m f (\bar x;\underbrace{0,0,\dots ,0}_{(m-1)-times})
\]
that $\ld m f (\bar x;\underbrace{0,0,\dots ,0}_{(m-1)-times};v)\ge0$ for all $v\in\E$. Therefore $\ld m f (\bar x;\underbrace{0,0,\dots ,0}_{(m-1)-times};d)=0$. Using similar arguments, we can prove that $\ld 1 f(\bar x;d)=0$. 
On the other hand, by Inequality (\ref{24}) we obtain that
\[
\ld n f (\bar x;\underbrace{0,0,\dots ,0}_{(n-1)-times};d)=\liminf_{t\downarrow 0,d^\pr\to d} n!\,t^{-n} [f(\bar x+t d^\pr)-f(\bar x)]
\ge\liminf_ {t\downarrow 0,d^\pr\to d}\,n!\,\alpha=n!\,\alpha>0,
\]
which is a contradiction.
\end{proof}

\begin{proposition}\label{pr2}
Let $x\in\E$. Then $\mathcal F_{n-1}(x)\subset\mathcal F_n(x)$ for every positive integer $n$, $n\ge 3$.
\end{proposition}
\begin{proof}
Suppose that there exists a function $f\in\mathcal F_{n-1}(x)$ with $f\notin\mathcal F_n(x)$. Therefore, $f$ satisfies Conditions (\ref{22}) and (\ref{23}), but it does not fulfil (\ref{24}). Then it follows from $f\in\mathcal F_{n-1}(x)$ that there exist positive numbers $\eps$, $\delta$ and $\alpha$ such that the following inequality  holds
\[
f(x+tu^\pr)\ge f(x)+\alpha t^{n-1},
\]
for all $t\in\R$ and $u^\pr\in\E$ with $0\le t<\delta$, $\norm{u^\pr-u}<\varepsilon$. Therefore
\[
\ld {n-1} f (x;\underbrace{0,0,\dots ,0}_{(n-2)-times};u)=\liminf_{t\downarrow 0,u^\pr\to u}\, (n-1)!\, t^{-(n-1)}\, [f(x+t u^\pr)-f(x)]
\]
\[
\ge\liminf_ {t\downarrow 0,u^\pr\to u}\,(n-1)!\,\alpha=(n-1)!\,\alpha>0,
\]
which contradicts the assumption 
$\ld {n-1} f (x;\underbrace{0,0,\dots ,0}_{(n-2)-times};u)=0$. Hence, in the case when $f\in\mathcal F_{n-1}(x)$,
(\ref{22}) and (\ref{23})  cannot be satisfied together. Therefore, the implication $\{(\ref{22}), (\ref{23})\Rightarrow (\ref{24})\}$ is fulfiled. Consequently $f\in\mathcal F_n(x)$.
\end{proof}

The following example shows that the inclusion in Proposition \ref{pr2} is strict:

\begin{example}
Consider the function $f:\R^2\to\R$ defined by
\[
f(x)=|x_1|^n+|x_2|^n.
\]
We prove that $f\in\mathcal F_n(\bar x)$, where $\bar x=(0,0)$. Let $u=(u_1,u_2)$ be a direction. We have 
\[
\ld 1 f(\bar x;u)=0\quad {\rm and}\quad (0,0)\in\lsubd 1 f(\bar x).
\]
 Therefore Conditions (\ref{22}) and (\ref{23}) hold. We prove that for every $\delta>0$ and each $\eps\in(0,1)$ there exists $\alpha>0$ which satisfies (\ref{24}); in other words there exists $\alpha>0$ such that $\alpha<|u_1^\pr|^n+|u_2^\pr|^n$, when 
\[
(u_1^\pr-u_1)^2+(u_2^\pr-u_2)^2<\eps^2\quad{\rm and }\quad u_1^2+u_2^2=1.
\]
 Assume the contrary that such positive number $\alpha$ does not exist. Therefore
\[
\inf\,\{|u_1^\pr|^n+|u_2^\pr|^n\mid(u_1^\pr-u_1)^2+(u_2^\pr-u_2)^2<\eps^2\}=0.
\]
Therefore, there exist infinite sequences of positive numbers $\{y_k\}$ and $\{z_k\}$ converging to 0 such that
\[
(y_k-u_1)^2+(z_k-u_2)^2<\eps^2,
\]
 which is impossible, because $\norm{u}=1$. Therefore $f\in\mathcal F_n(\bar x)$.

We prove that $f\notin\mathcal F_{n-1}(\bar x)$. Suppose the contrary that  $f\in\mathcal F_{n-1}(\bar x)$. Conditions (\ref{22}) and (\ref{23}) are satisfied when $x=(0,0)$. Suppose that (\ref{24}) is also fulfiled. Therefore, there exist $\delta>0$ and $\alpha>0$ such that $f(\bar x+tu)\ge f(\bar x)+\alpha t^{n-1}$, $0<t<\delta$. Hence
\[
t(|u_1|^n+|u_2|^n)>\alpha,\quad\forall t\in(0,\delta),
\]
which is obviously impossible.
\end{example}

\section{Comparison with some previous results}
\label{s8}
\setcounter{equation}{0}

In several papers Chaney introduced and studied a second-order directional derivative (see, for example, \cite{cha87}) which is called the derivative of Chaney. We recall its definition. 

It is called that a sequence $\{x_k\}$, $x_k\in\R^n$, $x_k\ne x$ converges to a point $x\in\R^n$ in direction $u\in\R^n$, $u\ne 0$ iff the sequence $\{(x_k-x)/\norm{x_k-x}\}$ converges to $u$. 

Let $f:\R^n\to\R$ be a locally Lipschitz function. Denote its Clarke generalized gradient at the point $x$ by $\partial f(x)$. Suppose that $u$ is a nonzero vector in $\R^n$. Denote by $\partial_u f(x)$ the set of all vectors $x^*$ such that there exist sequences $\{x_k\}$ and $\{x_k^*\}$ with $x_k^*\in\partial f(x_k)$, $\{x_k\}$ converges to $x$ in direction $u$, and $\{x_k^*\}$ converges to $x^*$. Really $\partial_u f(x)\subset\partial f(x)$.

\begin{definition}[\cite{cha87}]\label{def5}
Let $f:\R^n\to\R$ be a locally Lipschitz function. Suppose that $x\in\R^n$,  $u\in\R^n$, and $x^*\in\partial_u f(x)$. Then the second-order lower derivative of Chaney $\sld f(x;x^*;u)$ at $(x,x^*)$ in direction $u$ is defined to be the infimum of all numbers
\[
\liminf\,[f(x_k)-f(x)-x^*(x_k-x)]/t^2_k,
\]
taken over all triples of sequences $\{t_k\}$, $\{x_k\}$, and $\{x_k^*\}$ for which
\begin{enumerate} 
\item[{\rm (a)}] $t_k>0$ for each $k$ and $\{x_k\}$ converges to $x$,
\item[{\rm (b)}] $\{t_k\}$ converges to $0$ and $\{(x_k-x)/t_k\}$ converges to $u$,
\item[{\rm (c)}] $\{x_k^*\}$ converges to $x^*$ with $x_k^*\in\partial f(x_k)$ for each $k$.
\end{enumerate}
\end{definition}

The following claims due to Huang and Ng \cite[Theorems 2.2, 2.7 and 2.9]{hua94} are very important necessary and sufficient conditions for optimality in unconstrained optimization. The necessary conditions are generalizations of the respective results due to Chaney \cite[Theorem 1]{cha87} where the function is semismooth.

\begin{proposition}\label{pr4}
Let $\bar x$ be a local minimum point of the locally Lipschitz function $f$ and $u\in\R^n$ with norm $1$ such that $\dini f(\bar x;u)=0$. Then $0\in\partial_u f(\bar x)$, and
$\sld f(\bar x;0;u)\ge 0$.
\end{proposition}

\begin{proposition}\label{pr7}
Let  $f:\R^n\to\R$ be a locally Lipschitz function. Suppose that  $\dini f(x;v)\ge 0$, for all $v\in\R^n$.
For $u\in\R^n$ with norm $1$, if $\dini f(x;u)=0$, then $0\in\partial_u f(x)$.
\end{proposition}

\begin{proposition}\label{pr5}
Let  $f:\R^n\to\R$ be a locally Lipschitz function. Suppose that 
\[
\dini f(\bar x;v)\ge 0,\quad\forall v\in\R^n,\; v\ne 0.
\]
If $\sld f(\bar x;0;u)>0$ for all unit vectors $u\in\R^n$ for which $\dini f(\bar x;u)=0$, then $\bar x$ is a strict local minimizer.
\end{proposition}

\begin{lemma}\label{lema1}
Let $f:\R^n\to\R$ be a locally Lipschitz function. Suppose that $x\in\R^n$ and $u\in\R^n$. If $0\in\partial f_u(x)$ and $0\in\lsubd 1 f(x)$, then $\sld f(x;0;u)=\ld 2 f(x;0;u)$.
\end{lemma}
\begin{proof}
Denote $u_k=(x_k-x)/t_k$. Then
\[
\sld f(x;0;u)=\liminf\,[f(x+t_k u_k)-f(x)]/t^2_k,
\]
where the limes infimum is taken over all pairs of sequences $\{t_k\}$, $\{u_k\}$ which satisfy Conditions (a) and (b) from Definition \ref{def5}. It follows from here that
\[
\sld f(x;0;u)=\liminf_{t\downarrow 0,u^\pr\to u}\,[f(x+t u^\pr)-f(x)]/t^2=0.5\,\ld 2 f(x;0;u), \qedhere
\]
\end{proof}

\bigskip
{\it Proof of Proposition \ref{pr5} as corollary of Theorem \ref{th4}.}
 
Since $f$ is locally Lipschitz, then $\ld 1 f(\bar x;v)=\dini f(\bar x;v)\ge 0$ for all directions $v\in\R^n$. Therefore $0\in\lsubd 1 f(\bar x)$. Suppose that
$\dini f(\bar x;u)=0$ for some unit direction $u$. It follows from Lemma \ref{lema1} that 
\[
\ld 2 f(\bar x;0;u)=\sld f(\bar x;0;u)>0,
\]
 because by Proposition \ref{pr7} we have $0\in\partial_u f(\bar x)$. Then, according to Theorem \ref{th4} the point $\bar x$ is a strict local minimizer.
\qed 
\bigskip

It is seen that our proof is shorter than the proof of Huang and Ng \cite{hua94}.

\bigskip
{\it Proof of Proposition \ref{pr4} as corollary of Theorem \ref{th1}.}\;

Let $\dini f(\bar x;u)=0$ for some unit direction $u$. By Proposition \ref{pr7}, we have $0\in\partial_u f(\bar x)$.
Then, by Lemma \ref{lema1},  $\sld f(\bar x;0;u)=\ld 2 f(\bar x;0;u)$. Thus the claim follows from Theorem \ref{th1}.
\qed
\bigskip

Ben-Tal and Zowe introduced the following second-order derivative of a function $f:\E\to\R$ at the point $x\in\E$ in directions $u\in\E$ and $z\in\E$:
\[
f^{\pr\pr}_{BZ}(x;u,z):=\lim_{t\downarrow 0}\, t^{-2}[f(x+tu+t^2 z)-f(x)-t f^{\pr}(x;u)],
\]
where $f^{\pr}(x;u):=\lim_{t\downarrow 0}\, t^{-1}[f(x+tu)-f(x)]$ is the usual directional derivative of first-order. 

The following conditions are necessary for a local minimum in terms of the derivative of Ben-Tal and Zowe \cite{bt85}:
\begin{proposition}
Let $\bar x$ be a local minimizer of $f:\E\to\R$. Then
$$
f^\pr(\bar x;u)\ge 0,\quad\forall u\in\E,\eqno{\rm (BZ_1)}
$$
$$
f^\pr(\bar x;u)=0\quad\Rightarrow\quad f^{\pr\pr}_{BZ}(\bar x;u,z)\ge 0,\;\forall z\in\E.\eqno{\rm (BZ_2)}
$$
\end{proposition}

In the next result, we prove that Conditions (BZ$_1$) and (BZ$_2$) are consequence of (\ref{14}):
\begin{proposition}
Let $f:\E\to\R$ and $\bar x\in\E$ be a given function and a point respectively,  such that the derivatives $ f^\pr(\bar x;u)$ and $f^{\pr\pr}_{BZ}(\bar x;u,z)$ exist for all directions $u\in\E$ and $z\in\E$. Then Conditions (\ref{14}) imply that $({\rm BZ_1})$ and $({\rm BZ_2})$ are satisfied at $\bar x$.
\end{proposition}
\begin{proof}
Suppose that (\ref{14}) holds. Then the inequality $f^\pr(\bar x;u)\ge\ld 1 f(\bar x;u)\ge 0$ implies that (BZ$_1$) is satisfied. Let  $f^\pr(\bar x;u)=0$. Then the chain of relations
\[
f^{\pr\pr}_{BZ}(\bar x;u,z)=\lim_{t\downarrow 0} t^{-2}[f(\bar x+t(u+tz))-f(\bar x)]
\ge\liminf_{t\downarrow 0,u^\pr\to u} t^{-2}[f(\bar x+tu^\pr)-f(\bar x)]=
0.5\,\ld 2 f(\bar x;0;u)\ge 0 
\]
show that (BZ$_2$) is also satisfied for arbitrary $z\in\E$.
\end{proof}

Examples 2.1 and 2.2 in \cite{stu86} show that the optimality conditions given there cannot solve arbitrary set constrained problem.
The next example was given in \cite{stu86}:
\begin{example}
Consider the problem
\[
\inf\{f(x)\mid x\in C\},
\]
where the function $f:\R^2\to\R$ is defined by
\[
f(x_1,x_2)=\left\{
\begin{array}{ll}
|x_1|^n, & \textrm{if}\quad x_2\ne 0;\\
0, & \textrm{if}\quad x_2=0,
\end{array}\right.
\]
and $C=\R\times\{0\}$. Here $n$ is a positive integer. It is shown in \cite[Example 2.2]{stu86} that $\bar x=(0,0)$ is an isolated local minimizer of order $n$ over the set $C$, but the sufficient conditions in this paper cannot establish this fact, because the required derivatives are identical to $0$.

The derivatives that we study in the present work can solve the problem.
Consider the function $g$  such that
\[
g(x)=\left\{
\begin{array}{ll}
f(x), & \textrm{if}\quad x\in C;\\
+\infty, & \textrm{if}\quad x\notin C.
\end{array}\right.
\]
We obviously have 
\[
\bar x\in {\rm Argmin}\,\{f(x)\mid x\in C\}\quad\Longleftrightarrow\quad \bar x\in {\rm Argmin}\,\{g(x)\mid x\in\E\}.
\]
We prove that $\bar x$ is an isolated minimizer of order $n$. Let $u\in\R^2$ be an arbitrary direction such that $\norm{u}=1$.
Then $\ld 1 g(\bar x;u)=0$ and $\ld i g(\bar x;0,\dots,0;u)=0$ if $u=(\pm 1,0)$ for $i<n$. We have $\ld 1 g(\bar x;u)=+\infty$ and $\ld i g(\bar x;0,\dots,0;u)=+\infty$ if $u\ne (\pm 1,0)$ for $i<n$. Moreover, $0\in\lsubd 1 g(\bar x)$, $0\in\lsubd i g(\bar x,0,\dots,0)$ for $i<n$. At last, we obtain that $\ld n g(\bar x;0,\dots,0;u)=n!$ if $u=(\pm 1,0)$ and $\ld n g(\bar x;0,\dots,0;u)=+\infty$ if $u\ne (\pm 1,0)$.
 It follows from the sufficient conditions in Theorem \ref{th2} that $\bar x$ is an isolated local minimizer of order $n$.
\end{example}

Ginchev \cite{gin02} introduced the following directional derivatives of Hadamard type. Let be given an proper extended real function $\f$. The derivatives begin with the derivative of  order $0$:
\[
\gin 0 f(x;u):=\liminf_{t\downarrow 0,u^\pr\to u}\, f(x+t u^\pr).
\]
Let $n$ be a positive integer. Then the derivative of order $n$ is defined as follows: 
\[
\gin n f(x;u):=\liminf_{t\downarrow 0,u^\pr\to u}\,\frac{n!}{t^n}\,[f(x+t u^\pr)-\sum_{i=0}^{n-1}\frac{t^i}{i!}\gin i f(x;u)].
\]
In \cite[Theorem 1]{gin02} the author derived necessary optimality conditions of order $n$ for a local minimum under the assumption that the required derivatives exist. It is not discussed when  these derivatives exist.

The derivatives in \cite{aub90,gin02,stu86} are not consistent with the classical Fr\'echet derivatives. Really, for every twice Fr\'echet differentiable convex function 
$\gin 2 f(x;u)$ does not coincide with the second-order Fr\'echet directional derivative.  On the other hand, by Proposition \ref{pr3}, the derivatives introduced 
in Definition \ref{def3} are consistent with the classical ones. We can find only higher-order necessary conditions for a local minimum and second-order sufficient conditions for a weak local minimum in the book \cite[Section 6.6]{aub90}. 
\begin{example}
Consider the function $f:\R^2\to\R$ such that 
\[
f(x_1,x_2)=x_1^2+x_2^2.
\]
 Let $u=(u_1,u_2)$ be a direction. We have 
\[
\gin 0 f(x;u)=x_1^2+x_2^2,\quad \gin 1 f(x;u)=2x_1 u_1+2x_2 u_2.
\]
 Therefore,
\[
\gin 2 f(x;u)=2(u_1^2+u_2^2)+\liminf_{t\downarrow 0,u^\pr\to u}\,4 t^{-1}[x_1(u_1^\pr-u_1)+x_2(u_2^\pr-u_2)].
\]
If we take $t=t_k=1/k^2$, $u_2^\pr=u_2=0$, $u_1=1$, $u_1^\pr=1-1/k$, $x_1>0$, where $k$ is a positive integer with $k\to +\infty$, then we see that $\gin 2 f(x;u)=-\infty\ne\nabla^2 f(x)(u)(u)$.
\end{example}

This example also shows that the second-order contingent epi-derivative in \cite{aub90} is not consistent with the classical second-order derivative, and it equals $-\infty$ if we take the same values of the variables as in the example.


\begin{thebibliography}{00}

\bibitem{aub90} J.-P. Aubin and H. Frankowska, {\it Set-Valued Analysis}. Birkh\"auser, Boston, 1990.

\bibitem{aus84} A. Auslender, Stability in mathematical programming with
nondifferentiable data, {\it SIAM J. Control Optim.}, {\bf 22} (1984) 239--254.

\bibitem{pas08} D. Bedna\v r\'ik and K. Pastor, On second-order conditions in unconstrained optimization, {\it Math.~Program. Ser A} {\bf 11} (2008) 283--291.


\bibitem{bt85} A. Ben-Tal and J. Zowe, Directional derivatives in nonsmooth
optimization, {\it J. Optim. Theory Appl.} {\bf 47} (1985) 483--490.

\bibitem{bia96} M. Bianchi, Generalized quasimonotonicity and strong pseudomonotonicity of bifunctions, {\it Optimization} {\bf 36} (1996) 1--10.

\bibitem{cha87} R.W. Chaney, Second-order necessary conditions in constrained semismooth optimization, {\it SIAM J. Control Optim.} {\bf 25} (1987) 1072--1081.




\bibitem{die81} W.E. Diewert, Alternative characterizations of six kind of quasiconvexity in the nondifferentiable case with applications to nonsmooth programming, in {\it Generalized Concavity in Optimization and Economics}, S. Schaible and W.T. Ziemba (eds.),  Academic Press, New York, 1981, pp. 51--93.

\bibitem{d-a-z81} W.E. Diewert, M. Avriel and I. Zang, Nine kinds of quasiconcavity and concavity, {\it J. Econom. Theory}
{\bf 25} (1981) 397--420.

\bibitem{gin02} I. Ginchev, Higher order optimality conditions in nonsmooth optimization, {\it Optimization} {\bf 51} (2002) 47--72.

\bibitem{ban01} I. Ginchev and V.I. Ivanov, Higher order directional derivatives for nonsmooth functions, {\it C. R. Acad. Bulgare Sci.} {\bf 54} (2001) 33--38.


\bibitem{had93} N. Hadjisavvas and S. Schaible, On strong pseudomonotonicity
and (semi)strict quasimonotonicity, {\it J. Optim. Theory Appl.}, {\bf 79} (1995) 139--155.

\bibitem{han81} M.A. Hanson,  On sufficiency of the Kuhn-Tucker conditions, {\it J. Math. Anal. Appl.} {\bf 80} (1981) 545--550.

\bibitem{hof78} K.H. Hoffmann and H.J. Kornstaedt, Higher-order necessary conditions in abstract mathematical programming, {\it J. Optim. Theory Appl.} {\bf 26} (1978) 533--568.



\bibitem{hua94} L.R. Huang and K.F. Ng, Second-order necessary and sufficient conditions in nonsmooth optimization, {\it Math. Program.} {bf 66} (1994) 379--402.

\bibitem{il82} V.A. Il'in and E.G. Poznyak, {\it Fundamentals of Mathematical Analysis, Part 1}, Mir, Moscow, 1982 (Translated from Russian)

\bibitem{jogo2011} V.I. Ivanov, Second-order Kuhn-Tucker invex constrained problems, {\it J. Global Optim.} {\bf 50} (2011) 519--529.

\bibitem{optimization} V.I. Ivanov, Second-order invex functions in nonlinear programming, {\it Optimization} {\bf 61} (2012) 489--503.

\bibitem{jim08} B. Jimenez and V. Novo, Higher-order optimality conditions for strict local minima, {\it Ann. Oper. Res.} {\bf 157} (2008) 183--192.


\bibitem{lin82} A. Linneman, Higher-order necessary conditions for infinite and semi-infinite optimization, {\it J. Optim. Theory Appl.} {\bf 38} (1982) 483--511.

\bibitem{pal91} D. Pallaschke, P. Recht and R. Urbanski, Generalized derivatives for non-smooth functions, {\it Comment. Math. Prace Mat.} {\bf 31} (1991) 97--114.





\bibitem{stu86} M. Studniarski, Necessary and sufficient conditions for isolated local
minima of nonsmooth functions, {\it SIAM J. Control Optim.} {\bf 24} (1986) 1044--1049.

\bibitem{stu99} M. Studniarski and D. Ward, Weak sharp minima: characterizations and sufficient conditions, {\it SIAM J. Control Optim.} {\bf 38} (1991) 219--236. 


\bibitem{war94} D.E. Ward, Characterizations of strict local minima and necessary conditions for weak sharp minima, {\it J. Optim. Theory Appl.} {\bf 80} (1994) 551--571.

\end{thebibliography}
\end{document}